\documentclass[11pt]{article}
\pdfoutput=1

\usepackage[english]{babel}

\usepackage{amsmath,amssymb,amsthm}
\usepackage{tikz}
\usepackage{graphicx}

\usepackage[colorlinks=true, allcolors=blue]{hyperref}
\usepackage[capitalise,nameinlink]{cleveref}

\usepackage{amsfonts}

\usepackage{titling}

\usetikzlibrary{shapes.geometric}
\usetikzlibrary{graphs}
\usetikzlibrary{graphs.standard}
\usetikzlibrary{decorations.pathreplacing,calligraphy,backgrounds}

\newtheorem{theo}{Theorem}[section]

\newtheorem{cor}[theo]{Corollary}
\newtheorem{ques}[theo]{Question}

\newtheorem{lemma}[theo]{Lemma}

\usepackage{authblk}

\newcommand*\samethanks[1][\value{footnote}]{\footnotemark[#1]}

\author[1]{Zdeněk Dvořák\thanks{\url{rakdver@iuuk.mff.cuni.cz}. Supported by the ERC-CZ project LL2328 (Beyond the Four Color Theorem) of the Ministry of Education of Czech Republic.}}
\author[2]{Beatriz Martins\thanks{Supported by Projet ANR GODASse, Projet-ANR-24-CE48-4377.}}
\author[2]{Stéphan Thomassé\samethanks}
\author[2]{Nicolas~Trotignon\samethanks}
\affil[1]{Computer Science Institute, Charles University, Prague,\newline Czech Republic.}
\affil[2]{ENS de Lyon, CNRS, Université Claude Bernard Lyon 1,\newline LIP UMR 5668, 69342 Lyon Cedex 07, France.}

\begin{document}

\title{Lollipops, dense cycles and chords}

\date{\today}

\maketitle

\begin{abstract}
In 1980, Gupta, Kahn, and Robertson proved that every graph $G$ with minimum degree at least $k\geq 2$ contains a cycle $C$ containing  at least $k+1$ vertices each having at least $k$ neighbors in $C$ (so $C$ has at least $\frac{(k+1)(k-2)}{2}$ chords). In this work, we go further by showing that some of its edges can be contracted to obtain a graph with high minimum degree (we call such a minor of $C$ a \emph{cyclic minor}). We then investigate further graphs having cliques as cyclic minors, and show that minimum degree at least $O(k^2)$ guarantees a cyclic $K_k$-minor. 
\end{abstract}

\section{Introduction}
    Many theorems in graph theory state that a sufficiently high minimum degree in a graph guarantees the existence of some substructure that is in some sense dense, complex or well connected. We list below some classical examples:
    
    \begin{itemize}
        \item a highly connected subgraph~\cite{MADER72};
        \item a large clique as a minor~\cite{kos1982,DBLP:journals/dm/Vega83};
        \item a large clique as a topological minor~\cite{BBAT98};
        \item a large biclique as a subgraph or a subdivision of some prescribed graph as an induced subgraph~\cite{kuhn2004induced} and
        \item a $k$-linked subgraph~\cite{THOMAS2005309}.
    \end{itemize}

    Here, we add some items to this list, by exhibiting cycles that are dense in several ways as we explain now. 
    
    \subsection*{Many Chords}

    In~\cite{GUPTA198037}, the authors proved that if $G$ has minimum degree at least $k\geq 2$, then $G$ contains a cycle with at least $\frac{(k+1)(k-2)}{2}$ chords. An alternative proof to this was given in~\cite{kral03}. Here we refine the method used by~\cite{GUPTA198037} in order to obtain the following.

    \begin{theo}
        \label{th:main}
	If $G$ has minimum degree at least $k\geq 2$, then $G$ contains a cycle
	$C$ containing at least $k+1$ vertices each having at least $k$
	neighbors in $C$ (so $C$ has at least $\frac{(k+1)(k-2)}{2}$ chords).
    Moreover, in the graph  obtained by deleting all vertices not contained in $C$, there exist $X_1, X_2 \subseteq E(C)$ such that: 
    \begin{itemize}
	\item by contracting all edges in $X_1$, we obtain a graph of minimum degree at least $\left\lceil \frac{k+2}{2}\right\rceil$, and
	\item by contracting all edge in $X_2$, we obtain a graph of average degree at least $\tfrac{2}{3}(k+1)$.
	\end{itemize}
\end{theo}

    As mentioned in their work, some conclusions of \cite{GUPTA198037} are tight, which can be easily seen by considering a complete graph $K_t$, that has minimum degree $k=t-1$, and that obviously contains a cycle with exactly $k+1$ vertices of degree exactly~$k$ and exactly $\frac{(k+1)(k-2)}{2}$ chords.
    
    Now notice that, in order to do the contraction operation on the edges of $C$ and obtain a dense cycle $C'$, it is not enough to count the number of chords in~$C$. Indeed, $C$ may have many parallel chords, which can be a problem to obtain the dense contracted cycle $C'$. Here we use a similar method to the one used by~\cite{GUPTA198037}, however we refine it since we need to guarantee that there are many crossing chords in $C$.

 \subsection*{Dense Cyclic minors}
 
 Motivated by the edge contractions in Theorem~\ref{th:main}, we say that a graph $H$ is a \emph{cyclic minor} of a graph $G$ if a graph isomorphic to $H$ can be obtained from a Hamiltonian subgraph of $G$ by contracting some of the edges of the Hamiltonian cycle. So Theorem~\ref{th:main} just states that a large minimum degree guarantees a graph with high minimum degree as a cyclic minor. We will prove that by the Marcus-Tardos theorem~\cite{MARCUS2004153}, this cyclic minor can be further contracted in a cyclic way to form a complete bipartite graph (with some additional edges in the partite sets). 
 
 In fact, by using the notion of $k$-linked subgraph~\cite{THOMAS2005309}, it is not very difficult to prove that a large average degree guarantees a large complete graph as a cyclic minor, so that one may define  $f(\ell)$ as the smallest integer $\delta$ such that every graph of minimum degree at least $\delta$ contains $K_\ell$ as a cyclic minor.   We can prove that $f(4)=3$,  and $6 \le f(5)\le 8$ and more generally $f(\ell)=O(\ell^2)$.  We propose the following open question.

\begin{ques}Could it be that $f(\ell)=O(\ell\sqrt{\log \ell})$, matching the bound for standard minors from~\cite{kos1982,DBLP:journals/dm/Vega83}?
\end{ques}

 \subsection*{Lollipops}

    Our method to produce a dense cycle relies on \emph{lollipops}, that are subgraphs consisting of a path and a cycle containing a unique common vertex (a more formal definition is given below).  They were first defined and used by Thomason in~\cite{MR499124}. Since this seminal paper, the so-called \emph{lollipop method} has been extensively used, mostly to prove results about Hamiltonian cycles. We might cite about fifty papers citing~\cite{MR499124}, but we just mention one that has the advantage of being recent, more related to our topic and containing a short survey and nice results~\cite{DBLP:journals/jct/Thomassen18a}. We emphasize that in~\cite{GUPTA198037}, even though not with this terminology, the lollipop method was used to prove the existence of dense substructures. Here we use push this method further to obtain dense contracted cycles.
    
    \subsection*{Outline of the paper}

        In Section~\ref{sec:proofs}, we formally define optimal lollipops and prove by a sequence of lemmas that the cycle of such a lollipop satisfies all the properties given in Theorem~\ref{th:main} (namely, the theorem follows directly from Lemma~\ref{l:hamiltonian} and Lemma~\ref{l:contract}). In Section~\ref{sec:cyclic}, we compute $f$ for some values and prove the claims about the application of the Marcus-Tardos theorem~\cite{MARCUS2004153} and $k$-linked graphs.  In Section~\ref{sec:open}, we propose several conclusive remarks and open questions. 

    \subsection*{Definitions and notations}

    We mostly use standard terminology, see~\cite{diestel:graph}. It is convenient here to view a \emph{path} in a graph $G$ as a sequence of distinct vertices $P= p_1 \dots p_k$ such that for all $i\in \{1, \dots k-1\}$, $p_ip_{i+1} \in E(G)$.  Each edge $p_ip_{i+1}$ is an \emph{edge of~$P$} and any other edge between vertices of $P$ is a \emph{chord of $P$}. By $E(P)$ we denote the set of edges of $P$. We use a similar terminology for a cycle, that we view as a sequence $C = c_1 \dots c_kc_1$ of distinct vertices, where $k\geq 3$, and such that for all $i\in \{1, \dots k\}$, $c_ic_{i+1} \in E(G)$, with subscript taken modulo~$k$. The \emph{length} of a path (or cycle) is its number of edges. 
    When $P$ is a path and $a$ and $b$ are vertices of $P$, we denote by $aPb$ the subpath of $P$ from $a$ to $b$. 
    
    We use the notation $V(X)$ to denote the set of vertices of any kind of object $X$ that has vertices (so $V(G)$ for a graph $G$, $V(P)$ for a path $P$ and so on). We denote by $N(v)$ the set of neighbors of $v$ in a graph $G$ and we use the notation $N_X(v)$ for $N(v) \cap V(X)$ (again, $X$ can be a graph, a cycle and so on).  We denote by $d(v)$ the degree of a vertex $v$ (that is $|N(v)|$), and use the notation $d_X(v)$ for $|N_X(v)|$.

\section{Lollipops and dense cycles}
\label{sec:proofs}

A \emph{lollipop $L$ with path $P$ and cycle $C$} in a graph $G$ is a pair $(P, C)$ where $P = p_1\dots p_s$ ($s\geq 1$) is a path of $G$, $C = c_1\dots c_t c_1$ ($t\geq 3$) is a cycle of $G$, $p_s=c_1$ and $V(P) \cap V(C) = \{c_1\}$; see Fig.~\ref{fig_lollipop} for an illustration.

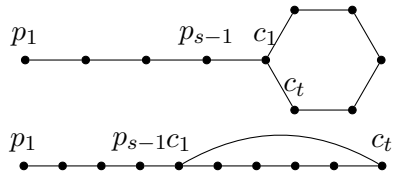
\begin{figure}[h!]
    \centering

\begin{tikzpicture}[scale=0.3]
  \tikzset{vertex/.style={circle, minimum size=0.1cm, fill=black, draw, inner sep=1pt}};

  \node[vertex, label={$p_1$}] (a) at (1.8bp,64.321bp) {};
  \node[vertex] (b) at (77.4bp,64.321bp) {};
  \node[vertex] (c) at (153.0bp,64.321bp) {};
  \node[vertex, label={$p_{s-1}$}] (d) at (228.6bp,64.321bp) {};
  \node[vertex,label={$c_1$}] (v) at (302.4bp,64.321bp) {};
  \node[vertex] (e) at (410.69bp,126.84bp) {};
  \node[vertex] (f) at (446.79bp,64.321bp) {};
  \node[vertex] (g) at (410.69bp,1.8bp) {};
  \node[vertex, label={$c_t$}] (w_t) at (338.5bp,1.8bp) {};
  \node[vertex, label={}] (w_1) at (338.5bp,126.84bp) {};

  \draw [] (a) ..controls (15.021bp,64.321bp) and (64.633bp,64.321bp)  .. (b);
  \draw [] (b) ..controls (90.621bp,64.321bp) and (140.23bp,64.321bp)  .. (c);
  \draw [] (c) ..controls (166.22bp,64.321bp) and (215.83bp,64.321bp)  .. (d);
  \draw [] (d) ..controls (241.41bp,64.321bp) and (289.09bp,64.321bp)  .. (v);
  \draw [] (e) ..controls (417.0bp,115.91bp) and (440.69bp,74.879bp)  .. (f);
  \draw [] (f) ..controls (440.47bp,53.387bp) and (416.78bp,12.358bp)  .. (g);
  \draw [] (g) ..controls (398.16bp,1.8bp) and (351.52bp,1.8bp)  .. (w_t);
  \draw [] (v) ..controls (308.71bp,75.254bp) and (332.4bp,116.28bp)  .. (w_1);
  \draw [] (w_1) ..controls (351.03bp,126.84bp) and (397.67bp,126.84bp)  .. (e);
  \draw [] (w_t) ..controls (332.18bp,12.733bp) and (308.5bp,53.762bp)  .. (v);

    \node[vertex, label={$p_1$}] (1) at (0,-2.4) {};
  \node[vertex] (2) at (1.71,-2.4) {};
  \node[vertex] (3) at (3.42,-2.4) {};
  \node[vertex, label={$p_{s-1}$}] (4) at (5.13,-2.4) {};
  \node[vertex,label={$c_1$}] (5) at (6.84,-2.4) {};
  \node[vertex] (6) at (8.55,-2.4) {};
  \node[vertex] (7) at (10.26,-2.4) {};
  \node[vertex] (8) at (11.97,-2.4) {};
  \node[vertex, label={}] (9) at (13.68,-2.4) {};
  \node[vertex, label={$c_t$}] (10) at (15.8,-2.4) {};

\draw (1)--(2)--(3)--(4)--(5)--(6)--(7)--(8)--(9)--(10);
\draw (10)edge[bend right](5);

\end{tikzpicture}

    \caption{Two representations of a lollipop $L$.}
    \label{fig_lollipop}

\end{figure}

A lollipop $L=(P, C)$ in $G$ is \emph{optimal} if: 
\begin{itemize}
    \item $L$ is vertex-wise maximal (that is no lollipop $L'$ of $G$ is such that $V(L) \subsetneq V(L'))$ and 
    \item among all lollipops on $V(L)$, $L$ has a cycle of maximum length (that is no lollipop $L' = (P', C')$ where $V(L')=V(L)$ is such that the length of $C'$ is greater than the length of $C$).  
\end{itemize}

\begin{lemma}
 \label{l:contain}
 Let $G$ be graph with minimum degree at least~2.  For every path $Q$ of $G$, there exists an (optimal) lollipop  that contains all vertices of $Q$. 
\end{lemma}

\begin{proof}
    Let $Q'= a\dots b$ be a vertex-inclusion-wise maximal path containing all vertices of~$Q$.  Since $d_G(b)\geq 2$ and $Q'$ is vertex-inclusion-wise maximal, $b$ is adjacent to at least two vertices of $Q'$, thus forming a lollipop.  Hence, if we take among all the lollipops on $V(Q')$ one that has a cycle with maximum length, then we obtain an optimal lollipop.
\end{proof}

\vspace{0.4cm}

From here on, we assume that $G$ is a graph with minimum degree~$k\geq 2$ and 
$L = (P, C)$ is an optimal lollipop of $G$ with notation as above. 

\begin{lemma}\label{lem:activeneighborhood}
 If $G[C]$ contains a Hamiltonian path with ends $c_1$ and $u$, then $N_G(u) \subseteq V(C)$. 
\end{lemma}

\begin{proof}
    Suppose for a contradiction that $Q = c_1\dots u$ is an Hamiltonian path of $G[C]$, $uv\in E(G)$ and $v\notin V(C)$. If $v\in V(P)$, then the lollipop $L'=(P', C')$ where $P'=p_1 P v$ and $C'= v P c_1 Q u v$ is a lollipop on $V(L)$ that has a cycle longer than $C$, a contradiction to the optimality of $L$. So $v\notin V(L)$. Then $L$ is not vertex-inclusion-wise maximal since by Lemma~\ref{l:contain} the path $P'=p_1Pc_1Quv$ is contained in a lollipop larger than $L$, a contradiction. Hence, $N_G(u) \subseteq V(C)$.
\end{proof}

\vspace{0.4cm}

Let us define recursively sets $\mathcal S_1, \mathcal S_2, \dots$ Each $\mathcal S_i$ is a set of Hamiltonian paths of $G[C]$ starting at $c_1$. 

\begin{itemize}
    \item    
$\mathcal S_1 = \{c_1 c_2 \dots c_t, c_1c_t c_{t-1} \dots c_2\}$.
\item 
For all $i\geq 1$, let us define $\mathcal S_{i+1}$ from $\mathcal S_{i}$.  For all paths $Q = c_1\dots u \in \mathcal S_i$ and all vertices $v$ such that $uv$ is a chord of $Q$, let $w$ be the neighbor of $v$ in $vQu$.  If $vw$ is an edge of $C$, then add the path $c_1QvuQw$ to $\mathcal S_{i+1}$ (see Fig.~\ref{fig_def_Si}). 
\end{itemize}

\begin{figure}[ht]
    \centering
    
    \begin{tikzpicture}[scale=0.3]

  \tikzset{vertex/.style={circle, minimum size=0.1cm, fill=black, draw, inner sep=1pt}};    

  \node[vertex, label={$c_1$}] (1'') at (0,3) {};
  \node[vertex] (2'') at (2,3) {};
  \node[vertex] (3'') at (4,3) {};
  \node[vertex, label={}] (4'') at (6,3) {};
  \node[vertex,label={$v$}] (5'') at (8,3) {};
  \node[vertex, label={$w$}] (6'') at (10,3) {};
  \node[vertex] (7'') at (12,3) {};
  \node[vertex] (8'') at (14,3) {};
  \node[vertex, label={$u$}] (9'') at (16,3) {};
  \node[vertex, color={white}, label={$\in \mathcal S_i$}] (cap1) at (18,2.7){};

\draw (1'')--(2'')--(3'')--(4'')--(5'')--(6'')--(7'')--(8'')--(9'');

  \node[vertex, label={$c_1$}] (1') at (0,-1) {};
  \node[vertex] (2') at (2,-1) {};
  \node[vertex] (3') at (4,-1) {};
  \node[vertex] (4') at (6,-1) {};
  \node[vertex,label={$v$}] (5') at (8,-1) {};
  \node[vertex, label={$w$}] (6') at (10,-1) {};
  \node[vertex] (7') at (12,-1) {};
  \node[vertex] (8') at (14,-1) {};
  \node[vertex, label={$u$}] (9') at (16,-1) {};

\draw (1')--(2')--(3')--(4')--(5')--(6')--(7')--(8')--(9');
\draw (9')edge[bend right](5');

  \node[vertex, label={$c_1$}] (1) at (0,-5) {};
  \node[vertex] (2) at (2,-5) {};
  \node[vertex] (3) at (4,-5) {};
  \node[vertex] (4) at (6,-5) {};
  \node[vertex,label={$v$}] (5) at (8,-5) {};
  \node[vertex, label={$w$}] (6) at (10,-5) {};
  \node[vertex] (7) at (12,-5) {};
  \node[vertex] (8) at (14,-5) {};
  \node[vertex, label={$u$}] (9) at (16,-5) {};
  \node[vertex, color={white}, label={$\in \mathcal S_{i+1}$}] (cap1) at (18,-5.3){};

\draw (1)--(2)--(3)--(4)--(5);
\draw (6)--(7)--(8)--(9);
\draw (9)edge[bend right](5);

    \end{tikzpicture}
    
  \caption{Construction of an element of $\mathcal S_{i+1}$ from an element of $\mathcal S_i$.}
    \label{fig_def_Si}
    
\end{figure}
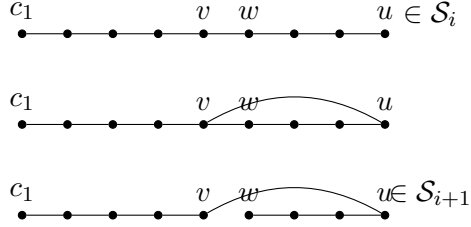

A path $Q$ of $G$ is said to be \emph{active} if $Q\in \mathcal{S}_i$. By extension, we also call \emph{active} every vertex $u\neq c_1$ that is an end of an active path. Note that $c_1$ is not active.  The following lemma is not needed, but we keep it since it illustrates the notion. 

\begin{lemma}
    For all $i\geq 1$, $\mathcal S_i \subseteq \mathcal S_{i+1}$.
\end{lemma}

\begin{proof}
    We make an induction on $i$. 
    For $i=1$, due to symmetry, is suffices to show that $c_1c_t c_{t-1} \dots c_2 \in\mathcal S_2$.  By definition, $Q = c_1 c_2 \dots c_t\in\mathcal S_1$. Since $c_1c_t$ is a chord of $Q$, it is possible to take the Hamiltonian path $Q'=c_1c_t c_{t-1} \dots c_2$, so $Q' \in \mathcal S_2$. Hence, $\mathcal S_1\subseteq\mathcal S_2$.

    For $i>1$, let $Q \in\mathcal S_i$, where $Q=c_1\ldots u$. So there is a path $Q' = c_1 \dots u' \in\mathcal S_{i-1}$ such that $vu'$ is a chord of $Q'$ and $u$ is the neighbor of $v$ in $vQ'u'$. Moreover $vu$ is an edge of $C$.  By the induction hypothesis, since $Q'\in\mathcal S_{i-1}$, we also have $Q'\in\mathcal \mathcal S_i$. By considering this path $Q'$, the definition of $S_{i+1}$ then implies that $Q\in \mathcal S_{i+1}$.
\end{proof}

As mentioned in the introduction, Lemma~\ref{lem:kactive} and Lemma~\ref{l:hamiltonian} where already proved in \cite{GUPTA198037}; however, the arguments used are not enough to provide the structure wanted in this paper. Specifically being the end an Hamiltonian path is enough to guarantee a high degree, but not enough to ensure the edges contractions used later. Hence, we have to keep the proofs of the next two lemmas, even tough they look similar to the ones in \cite{GUPTA198037}.

\begin{lemma}\label{lem:kactive}
    $C$ contains at least $k$ active vertices. Moreover, if $d_C(c_1) < k$, then $C$ contains at least $k+1$ active vertices.   
\end{lemma}

\begin{proof}
    By definition $Q =c_1c_2 c_3 \dots c_t \in \mathcal S_1$, and thus the vertex $c_t$ is active and $N_G(c_t)\subseteq V(C)$ by~Lemma~\ref{lem:activeneighborhood}. So, there exist integers $2 \leq i_1 < \dots < i_{k-2} \leq t-2$ such that  $c_{i_j}\in N_G(c_t)$ for $j=1, \dots, k-2$. Since $c_tc_{i_j}$ is a chord of $Q$ and $c_{i_j}c_{i_j+1}$ is an edge of $C$, we have $Q_j = c_1Qc_{i_j}c_tQc_{i_j+1} \in\mathcal S_2$, so $c_{i_j+1}$ is active.  Now, $c_2, c_{i_1+1}, \dots, c_{i_{k-2}+1}, c_t$ are $k$ distinct active vertices, proving the first conclusion. 
     
   Now assume that $d_C(c_1)<k$. So, $c_1$ cannot be adjacent to all vertices in $\{c_2, c_{i_1+1}, \dots, c_{i_{k-2}+1}, c_t\}$. Hence, there exists an integer $1\leq j \leq k-2$ such that $c_1c_{{i_j}+1}\notin E(G)$, $c_tc_{i_j}\in E(G)$, $c_{{i_j}+1}$ is active and $R\in \mathcal S_2$ where $R=c_1Qc_{i_j}c_tQc_{{i_j}+1}$. Since $c_{{i_j}+1}$ is active and $c_1c_{{i_j}+1}\notin E(G)$, by~Lemma~\ref{lem:activeneighborhood}, $N_G(c_{{i_j}+1})\subseteq V(C)\setminus\{c_1\}$. Among the neighbors of $c_{{i_j}+1}$, some are in $Q' = c_2 Q c_{{i_j}-1}$, and some are in $Q''=c_{{i_j}+3} Qc_t$.  
   So, there exist integers $k' = |N_G(c_{{i_j}+1})\cap V(Q')|$ and $k'' = |N_G(c_{{i_j}+1})\cap V(Q'')|$, $2 \leq i'_1 < \dots < i'_{k'} \leq {i_j}-1$ and ${i_j}+3 \leq i''_1 < \dots < i''_{k''} \leq t$ such that $c_{i'_h} \in N_{Q'}(c_{{i_j}+1})$ for $h=1, \dots, k'$ and $c_{i''_h} \in N_{Q''}(c_{{i_j}+1})$ for $h=1, \dots, k''$.  Note that $k'+k'' \ge k-2$. 

\begin{figure}[ht]
\centering

\begin{tikzpicture}[scale=0.55]
  \tikzset{vertex/.style={circle, minimum size=0.2cm, fill=black, draw, inner sep=1pt}};

  \node[vertex, label={$p_1$}] (a) at (1.8bp,116.23bp) {};
  \node[vertex] (b) at (77.4bp,116.23bp) {};
  \node[vertex] (c) at (153.0bp,116.23bp) {};
  \node[vertex, label={$p_{s-1}$}] (d) at (228.6bp,116.23bp) {};
  \node[vertex, label={$c_1$}] (v) at (302.4bp,116.23bp) {};
  \node[vertex, label={$c_{i'_h}$}] (e) at (385.54bp,230.66bp) {};
  \node[vertex, label={$c_{i'_h+1}$}] (f) at (459.9bp,230.66bp) {};
  \node[vertex] (g) at (520.06bp,186.96bp) {};
  \node[vertex, label=right:{$c_{i_j}$}] (1) at (543.04bp,116.23bp) {};
  \node[vertex, label=below:{$c_{i_j+1}$}] (2) at (520.06bp,45.509bp) {};
  \node[vertex] (3) at (459.9bp,1.8bp) {};
  \node[vertex] (4) at (385.54bp,1.8bp) {};
  \node[vertex, label=left:{$c_t$}] (w_t) at (325.38bp,45.509bp) {};
  \node[vertex, label={$c_2$}] (w_1) at (325.38bp,186.96bp) {};

  \draw [] (a) ..controls (15.021bp,116.23bp) and (64.633bp,116.23bp)  .. (b);
  \draw [] (b) ..controls (90.621bp,116.23bp) and (140.23bp,116.23bp)  .. (c);
  \draw [] (c) ..controls (166.22bp,116.23bp) and (215.83bp,116.23bp)  .. (d);
  \draw [] (d) ..controls (241.41bp,116.23bp) and (289.09bp,116.23bp)  .. (v);
  \draw [] (e) ..controls (398.45bp,230.66bp) and (446.49bp,230.66bp)  .. (f);
  \draw [] (f) ..controls (470.42bp,223.02bp) and (509.9bp,194.34bp)  .. (g);
  \draw [] (g) ..controls (524.08bp,174.59bp) and (539.16bp,128.18bp)  .. (1);
  \draw [] (1) ..controls (539.02bp,103.86bp) and (523.94bp,57.453bp)  .. (2);
  \draw [] (2) ..controls (509.54bp,37.865bp) and (470.06bp,9.1815bp)  .. (3);
  \draw [] (3) ..controls (446.99bp,1.8bp) and (398.95bp,1.8bp)  .. (4);
  \draw [] (4) ..controls (375.02bp,9.4437bp) and (335.54bp,38.128bp)  .. (w_t);
  \draw [] (v) ..controls (306.42bp,128.6bp) and (321.5bp,175.01bp)  .. (w_1);
  \draw [] (w_1) ..controls (335.9bp,194.6bp) and (375.38bp,223.28bp)  .. (e);
  \draw [] (w_t) ..controls (321.36bp,57.877bp) and (306.28bp,104.29bp)  .. (v);
  \draw[]{} (w_t) -- (1);
  \draw[]{}(2) -- (e);
  \draw[dotted]{} (v) -- (2);
\end{tikzpicture}

\caption{A situation from the proof of \cref{lem:kactive}.}
\label{fig_biglollipop1}

\end{figure}
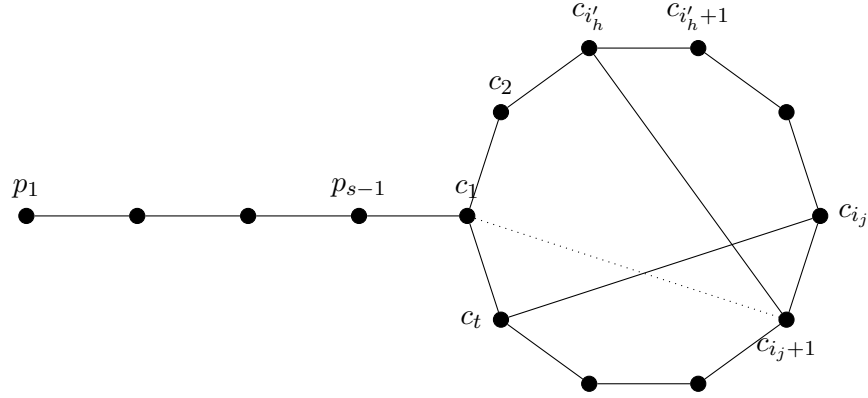

   For all $h = 1, \dots, k'$, the path $$R'_h = c_1 R c_{i'_h} c_{i_j+1} R c_{i'_h+1} =  c_1 Q c_{i'_h} c_{{i_j}+1} Q c_t c_{i_j} Q c_{i'_h+1}$$ is in $\mathcal S_3$ and shows that $c_{i'_h+1}$ is active (see~Fig.~\ref{fig_biglollipop1}).
   
   For all $h=1, \dots, k''$, the path $$R''_h = c_1 R  c_{i''_h} c_{{i_j}+1} R c_{i''_h-1} = c_1 Q c_{i_j} c_t Q c_{i''_h} c_{{i_j}+1} Q c_{i''_h-1}$$ is in $\mathcal S_3$ and shows that $c_{i''_h-1}$ is active (see~Fig.~\ref{fig_biglollipop2}).
   
   \begin{figure}[ht]
\centering

\begin{tikzpicture}[scale=0.55]
  \tikzset{vertex/.style={circle, minimum size=0.2cm, fill=black, draw, inner sep=1pt}};

  \node[vertex, label={$p_1$}] (a) at (1.8bp,116.23bp) {};
  \node[vertex] (b) at (77.4bp,116.23bp) {};
  \node[vertex] (c) at (153.0bp,116.23bp) {};
  \node[vertex, label={$p_{s-1}$}] (d) at (228.6bp,116.23bp) {};
  \node[vertex, label={$c_1$}] (v) at (302.4bp,116.23bp) {};
  \node[vertex] (e) at (385.54bp,230.66bp) {};
  \node[vertex, label={$c_{i_j}$}] (f) at (459.9bp,230.66bp) {};
  \node[vertex, label={$c_{i_j+1}$}] (g) at (520.06bp,186.96bp) {};
  \node[vertex] (1) at (543.04bp,116.23bp) {};
  \node[vertex, label=below:{$c_{i''_h-1}$}] (2) at (520.06bp,45.509bp) {};
  \node[vertex, label=below:{$c_{i''_h}$}] (3) at (459.9bp,1.8bp) {};
  \node[vertex] (4) at (385.54bp,1.8bp) {};
  \node[vertex, label=left:{$c_t$}] (w_t) at (325.38bp,45.509bp) {};
  \node[vertex, label={$c_2$}] (w_1) at (325.38bp,186.96bp) {};

  \draw [] (a) ..controls (15.021bp,116.23bp) and (64.633bp,116.23bp)  .. (b);
  \draw [] (b) ..controls (90.621bp,116.23bp) and (140.23bp,116.23bp)  .. (c);
  \draw [] (c) ..controls (166.22bp,116.23bp) and (215.83bp,116.23bp)  .. (d);
  \draw [] (d) ..controls (241.41bp,116.23bp) and (289.09bp,116.23bp)  .. (v);
  \draw [] (e) ..controls (398.45bp,230.66bp) and (446.49bp,230.66bp)  .. (f);
  \draw [] (f) ..controls (470.42bp,223.02bp) and (509.9bp,194.34bp)  .. (g);
  \draw [] (g) ..controls (524.08bp,174.59bp) and (539.16bp,128.18bp)  .. (1);
  \draw [] (1) ..controls (539.02bp,103.86bp) and (523.94bp,57.453bp)  .. (2);
  \draw [] (2) ..controls (509.54bp,37.865bp) and (470.06bp,9.1815bp)  .. (3);
  \draw [] (3) ..controls (446.99bp,1.8bp) and (398.95bp,1.8bp)  .. (4);
  \draw [] (4) ..controls (375.02bp,9.4437bp) and (335.54bp,38.128bp)  .. (w_t);
  \draw [] (v) ..controls (306.42bp,128.6bp) and (321.5bp,175.01bp)  .. (w_1);
  \draw [] (w_1) ..controls (335.9bp,194.6bp) and (375.38bp,223.28bp)  .. (e);
  \draw [] (w_t) ..controls (321.36bp,57.877bp) and (306.28bp,104.29bp)  .. (v);
  \draw[]{} (w_t) -- (f);
  \draw[]{}(3) -- (g);
  \draw[dotted]{} (v) -- (g);
\end{tikzpicture}
\caption{Another situation from the proof of \cref{lem:kactive}.}
\label{fig_biglollipop2}

\end{figure}
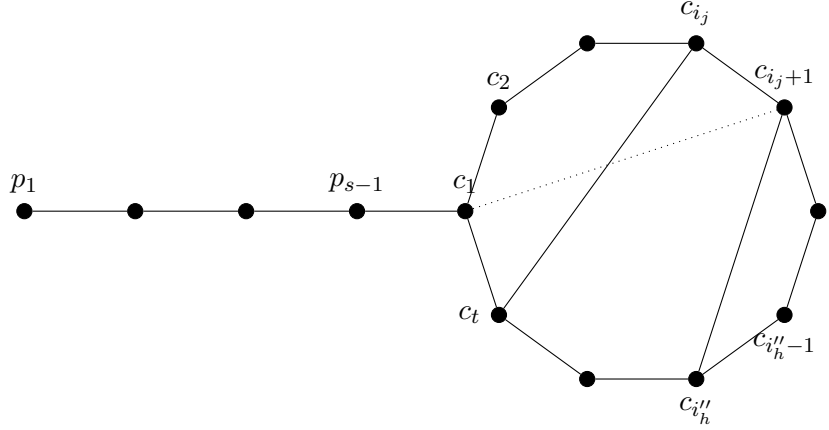
   
   Hence, $c_2,$ $c_{i'_1+1},$ $\dots,$ $c_{i'_{k'}+1},$ $c_{{i_j}+1},$ $c_{i''_1-1},$ $\dots,$ $c_{i''_{k''}-1},$ $c_t$, are $k'+k''+3\ge k+1$ distinct active vertices, proving the second conclusion.
\end{proof}

\begin{lemma}\label{l:hamiltonian}
     $C$ contains at least $k+1$ vertices of degree at least $k$ (in $C$). In particular, $C$ has at least $\frac{(k+1)(k-2)}{2}$ chords.
\end{lemma}

\begin{proof}
    Indeed, by Lemma~\ref{lem:kactive}, $C$ contains at least $k+1$ active vertices if $d_C(c_1)<k$. Moreover, if $d_C(c_1)\geq k$, then $C$ has $k$ active vertices. In both cases, by Lemma~\ref{lem:activeneighborhood}, $C$ has at least $k+1$ vertices of degree at least $k$ in~$C$. The fact that $C$ has at least $\frac{(k+1)(k-2)}{2}$ chords follows since every one of the $k+1$ vertices provides $k-2$ chords of $C$ and the number is divided by two to avoid double counting.
\end{proof}

An edge of $C$ whose both ends are non-active is \emph{passive}.

\begin{lemma}\label{lem:allpath}
All active paths go through all passive edges of $C$.
\end{lemma}

\begin{proof}
 Let $Q$ be an active path and $e$ be a passive edge of $C$.  By definition of active paths, there exist an integer $i\geq 1$ such that $Q \in \mathcal S_i$. Let us prove by induction on $i$ that $Q$ goes through $e$. 

 If $i=1$, then $Q$ is one of $c_1 c_2 \dots c_t$ or $c_1c_t c_{t-1} \dots c_2$. Note that $ e\notin \{c_1c_2, c_1c_t\}$ because $c_2$ and $c_t$ are both active. So $Q$ goes through all passive edges of $C$, in particular through $e$. 

 Now suppose $i>1$.  Since $Q=c_1\ldots u\in \mathcal S_i$, by definition of $\mathcal S_i$, there exists an active path $Q'= c_1 \dots u' \in \mathcal S_{i-1}$ such that $Q'$ has a chord $u'v$ and $u$ is the neighbor of $v$ in $vQ'u'$.  By the induction hypothesis, $Q'$ goes through $e$.  So does $Q$, since $e\neq uv$ (because $u$ is active). 
\end{proof}

\begin{lemma}\label{lem:nonactivepath}
    If $R$ is a subpath of $C$ containing only passive edges and $u$ is an active vertex, then $u$ has at most one neighbor in $R$.  Moreover, such a neighbor is an end of $R$. 
\end{lemma}

\begin{proof}
    Since $u$ is active, there exists a path $Q = c_1 \dots  u$ in some $\mathcal S_i$. 
By Lemma~\ref{lem:allpath}, $R$ is a subpath of $Q$, so let $a_1, a_2, \ldots, a_r=V(R)$ be the consecutive vertices of $Q$, ordered in a such a way that $c_1, a_1, a_2, \ldots, a_r, u$ appear in this order along $Q$.  
If $u$ is adjacent to $a_j$, where $1\leq j < r$, then $ua_j$ is a chord of $Q$ and $a_ja_{j+1}$ is an edge of $C$. Hence, the path $Q'=c_1 Q a_j u Q a_{j+1}$ is in $\mathcal S_{i+1}$ and the vertex $a_{j+1}$ is  active, a contradiction to $a_ja_{j+1}$ being passive.  So, $a_r$ is the only possible neighbor of $u$ in $R$. 
\end{proof}

\begin{lemma}
    \label{l:contract}
    There exist sets $X_1,X_2\subseteq E(C)$ so that for $i\in\{1,2\}$, the graph $G_i$ obtained by deleting all vertices not in $C$ and contracting the edges of $X_i$
    has at least $k$ vertices and
    \begin{itemize}
        \item minimum degree at least $\left\lceil \frac{k+2}{2}\right\rceil$ if $i=1$, and 
        \item average degree at least $\tfrac{2}{3}(k+1)$ if $i=2$.
    \end{itemize}
\end{lemma}

\begin{proof}
    By~Lemma~\ref{lem:kactive}, $C$ has at least $k$ active vertices. Let $X_0$ be the set of 
    passive edges of $C$.  Let $G_0$ be the graph obtained from $G$ by deleting the vertices outside of $C$
    and contracting the edges of $X_0$; and let $C_0$ be the cycle in $G$ obtained from $C$ by this contraction.
    By~Lemma~\ref{lem:activeneighborhood} and Lemma~\ref{lem:nonactivepath}, the active vertices have the same degree in $G_0$ as in $G$.

    Note that the cycle $C_0$ is edgewise partitioned into edges whose both ends are active vertices, and
    paths of length~2 whose both ends are active and whose unique internal
    vertex is not.  Let $a_1b_1c_1$, \dots, $a_pb_pc_p$ be the paths
    of length~2 of $C_0$ such that the $a_i$'s and $c_i$'s are active and the
    $b_i$'s are not.  Suppose that $a_1, b_1, c_1$, \dots, $a_p, b_p, c_p$
    appear in this order along $C_0$ (note that possibly $c_i=a_{i+1}$,
    subscript taken modulo $p$).

    Let $X'_1=\{a_ib_i:i\in \{1,\ldots,p\}\}$ and $X_1=X_0\cup X'_1$; thus, $G_1$ is the graph obtained from $G_0$ by contracting the edges in $X'_1$.
    Note that the edges of $X'_1$ are vertex-disjoint.   Let $C_1$ be the cycle in $G_1$ obtained from $C_0$ by this contraction.
    Active vertices are in one-to-one correspondence to the vertices of $C_1$; in particular, $C_1$ has length at least $k$.
    Moreover, each vertex of $C_1$ has its two neighbors along $C_1$, and it is incident with at least $\left\lceil \frac{k-2}{2}\right\rceil$
    chords (because contracting the edges of $X'_1$ decreases the number of incident chords
    at most by half).  Therefore, $G_1$ has minimum degree at least $\left\lceil \frac{k+2}{2}\right\rceil$.

    We let $X_2=X_1$ if $G_1$ has average degree greater than $G_0$ and $X_2=X_0$ otherwise.  Thus $G_2 = G_1$ if $X_2=X_1$, and $G_2 = G_0$ otherwise.   Let $n_a$ be the number of edges in $E(G_0)\setminus E(C_0)$
    with both active ends, let $n_b$ be the number of edges in $E(G_0)\setminus E(C_0)$ with exactly one active end, and let $m$ be the number of active vertices.
    Since each active vertex is incident with at least $k-2$ chords in $G_0$,
    we have
    $$2n_a+n_b\ge (k-2)m.$$
    The average degree of $G_0$ is at least
    $$\frac{2|E(G_0)|}{|V(G_0)|}\ge \frac{2|C_0|+2n_a+2n_b}{|C_0|}=2+\frac{2n_a+2n_b}{|C_0|}\ge 2+\frac{n_a+n_b}{m}.$$
    The graph $G_1$ has at least $n_a$ chords and $m$ vertices, and thus it has average degree at least
    $$\frac{2|E(G_1)|}{|V(G_1)|}\ge 2+\frac{2n_a}{m}.$$
    Hence, $G_2$ has average degree at least
    $$2+\frac{\max(n_a+n_b,2n_a)}{m}\ge 2+\frac{\max((k-2)m-n_a,2n_a)}{m}\ge 2+\tfrac{2}{3}(k-2)=\tfrac{2}{3}(k+1).$$
\end{proof}

\section{Cyclic minors}
\label{sec:cyclic}

In Lemma~\ref{l:contract}, we showed that the lollipop cycle $C$  (obtained in the proof of Theorem~\ref{th:main}) admits a cyclic minor $M$ with vertices $x_1,x_2,\dots ,x_m$ (enumerated in the cyclic order) with average degree more than $2k/3$. Let us push this argument to show that we can further contract $M$ in a cyclic way to obtain a large complete bipartite graph. To see this, we first remind the statement of a celebrated theorem of Marcus and Tardos~\cite{MARCUS2004153}. 

\begin{theo}
\label{th:mt}
For every integer $a$, there exists an integer $c_a$ such for every $n\times n$ $(0-1)$-matrix $M$ containing at least $c_a n$ 1-entries, there exists a partition of $M$ into $a\times a$ blocks such that each block contains a 1-entry. 
\end{theo}

Note that in~\cite{MARCUS2004153}, Theorem~\ref{th:mt} is presented slightly differently. The existence of a number $f(n, P)$ is proved for all integers $n$ and all permutation matrices $P$, with the following property: $f(n, P)$ is the maximum number of 1-entries in an $n \times n$ matrix that does not contain $P$ as a submatrix.  Furthermore, Theorem~1 in ~\cite{MARCUS2004153} states that $f(n, P) = O(n)$.  Theorem~\ref{th:mt} for $a$ is a direct consequence of this statement, with a specific $a^2 \times a^2$ matrix~$P$, see Fig.~\ref{f:P} for $a=3$.

\begin{figure}
\[
\left[
\begin{array}{ccc|ccc|ccc}
1 & 0 & 0 & 0 & 0 & 0 & 0 & 0 & 0 \\
0 & 0 & 0 & 1 & 0 & 0 & 0 & 0 & 0 \\
0 & 0 & 0 & 0 & 0 & 0 & 1 & 0 & 0 \\
\hline
0 & 1 & 0 & 0 & 0 & 0 & 0 & 0 & 0 \\
0 & 0 & 0 & 0 & 1 & 0 & 0 & 0 & 0 \\
0 & 0 & 0 & 0 & 0 & 0 & 0 & 1 & 0 \\
\hline
0 & 0 & 1 & 0 & 0 & 0 & 0 & 0 & 0 \\
0 & 0 & 0 & 0 & 0 & 1 & 0 & 0 & 0 \\
0 & 0 & 0 & 0 & 0 & 0 & 0 & 0 & 1
\end{array}
\right]
\]
\caption{A permutation matrix}\label{f:P}
\end{figure}

Now, in order to prove the next results, given a matrix $B$, we define a sub-matrix $B_{n_1:n_2,m_1:m_2}$ as $(b_{i,j})_{\begin{smallmatrix}n_1\leq i\leq n_2 \\ m_1\leq j\leq m_2\end{smallmatrix}}$. Furthermore, we denote by $K'_{\ell,\ell}$ the graph obtained from $K_{\ell,\ell}$ by adding a path on each of its partite set.

\begin{theo} 
 For every integer $\ell$, there exists an integer $k$ such that every graph with minimum degree at least $k$ contains $K'_{\ell, \ell}$ as a cyclic minor.
 \end{theo}

\begin{proof}
  Set $a= 2\ell$ and apply Theorem~\ref{th:mt}. Consider the smallest integer $k$ such that $2k/3 \geq c_a$ (note that $k$ depends only on $\ell$). Now consider a graph with minimum degree at least $k$. It contains a cyclic minor $M$ with vertices $v_1,v_2,\dots, v_m$ (enumerated in the cyclic order) and average degree at least $2k/3 \geq c_a$ by Lemma~\ref{l:contract}. Let $B$ be the $m\times m$ $(0-1)$-adjacency matrix of $M$ and denote by $b_{i,j}$ the element in the $i$-th line and $j$-th column of $B$. Notice that $B$ has at least $2km/3 \geq c_a m$ entries~$1$. Hence, by Theorem~\ref{th:mt}, there is a partition of $B$ into $(2\ell)\times(2\ell)$ blocks such that each block contains an entry~1. 
  So, there exist integers $0 = i_0 < \dots < i_{2\ell} = m$ and $0 = j_0 < \dots < j_{2\ell} = m$ such that for all $x, y \in \{ 0, \dots, 2\ell-1\}$, $B_{i_x+1:i_{x+1},j_y+1:j_{i+1}}$ contains a 1. 
  
  Up to symmetry, we may assume $i_\ell \leq j_\ell$. 
  Note that the sets of vertices $\{v_1, \dots, v_{i_\ell}\}$ and $\{v_{j_\ell+1}, \dots, v_{j_m}\}$ are disjoint. Now set $X_{x+1} = \{v_{i_x+1}, \dots, v_{i_{x+1}}\}$ for all $x\in \{0, \dots, \ell-1\}$ and  $Y_{y-\ell+1} = \{v_{j_y+1}, \dots, v_{j_{y+1}}\}$ for all $y\in \{\ell, \dots, 2\ell-1\}$.  These sets are pairwise disjoint because $i_\ell \leq j_\ell$.  And there is an edge between any $X_x$ and any $Y_y$ because each block $B_{i_x+1:i_{x+1},j_y+1:j_{i+1}}$ contains a 1.
  By contracting each  set $X_1, \dots, X_\ell$ and $Y_1, \dots, Y_\ell$, and  furthermore contract the vertices in $\{v_{i_\ell+1} \dots, v_{j_\ell}\}$ (if any) with $Y_1$, we therefore obtain $K'_{\ell, \ell}$. 
\end{proof}

Knowing this result, a natural question is then to try to find complete graphs as cyclic minors of $C$. In order to do it, we recall the definition of $f(\ell)$, which is the smallest integer $\delta$ such that every graph of minimum degree at least $\delta$ contains $K_\ell$ as a cyclic minor. Let us start our investigation with small cliques as cyclic minors:

\begin{lemma}\label{l:clique}
We have $f(3)=2$, $f(4)=3$ and $f(5) \leq 8$. 
\end{lemma}

\begin{proof}
It is clear that if $\delta(G)\geq 2$, then $G$ has a cycle $C$ that can have its edges contracted in order to obtain a $K_3$.

 By Theorem~\ref{th:main}, a graph of minimum degree at least~$3$ has a cyclic minor of minimum degree at least $3$, and a graph of minimum degree at least eight has a cyclic minor of average degree at least six.
Suppose first that a graph $G$ has minimum degree at least three, hence there is a cyclic minor $F$ of $G$ with minimum degree at least $3$. Let $C$ be a Hamiltonian cycle of $F$. Choose a chord $uv$ of $C$ and a subpath $P$ of $C$ with ends $u$ and $v$ so that the path $P$ is as short as possible.
Let $z\in V(P)\setminus\{u,v\}$ be an arbitrary vertex.  Since $F$ has minimum degree at least three,
$z$ is incident with a chord $zx$, and by the minimality of $|E(P)|$, we have $x\not\in E(P)$.
Hence, there exist pairwise vertex-disjoint subpaths $P_1$, \ldots, $P_4$ of $C$ such that
$V(C)=V(P_1)\cup\ldots\cup V(P_4)$, $u\in V(P_1)$, $z\in V(P_2)$, $v\in V(P_3)$, and $x\in V(P_4)$.
Contracting the edges of these paths turns $F$ into $K_4$.

Suppose next that $G$ is a graph such that $\delta(G)= 8$, hence it has a cyclic minor $F$ with average degree at least $6$, and thus $|E(F)|\ge 3|V(F)|$. As in the previous paragraph, consider $C$ a Hamiltonian cycle of $F$.  Without loss of generality, we can assume that
for any edge $e\in E(C)$, contracting $e$ results in a graph of average degree less than $6$.
Hence,
$$|E(F/e)|\le 3|V(F/e)|-1=3|V(F)|-4\le |E(F)|-4,$$
and thus $e$ is contained in at least three triangles in $F$.  Let $z_1$, \ldots, $z_t$ with $t\ge 3$
be the common neighbors of the ends of $e$ in $F$, in order along the path $C-e$.  We say that the vertices $z_2$, \ldots, $z_{t-1}$
are the \emph{peaks} for $e$.

Let $e=uv$ be an edge of $C$, $z$ a peak of $e$, and $P$ a path in $C-e$ from $z$ to $u$ or $v$ chosen so that $P$ is as short as possible.  By symmetry, we can assume that $u$ is an end of $P$.  Note that $u$ and $v$ have a common neighbor $z_1$ in $P-z$, since $z$ is a peak for $e$.  Let $v'$ be the neighbor of $u$ in $P$ ($v'=z_1$ is possible). Let $z'$ be a peak of the edge $uv'$.  By the minimality of $E(P)$, we have $z'\not\in V(P)$, and since $z'$ is a peak for $uv'$, we have $z'\neq v$.  Hence, $z'$ is a vertex of the path $P'=C-(V(P)\cup\{v\})$.  Observe that contracting the edges of the paths $P-\{u,z\}$ and $P'$ turns $F$ into $K_5$ (see~Fig.~\ref{fig_cyclick5}).
\end{proof}

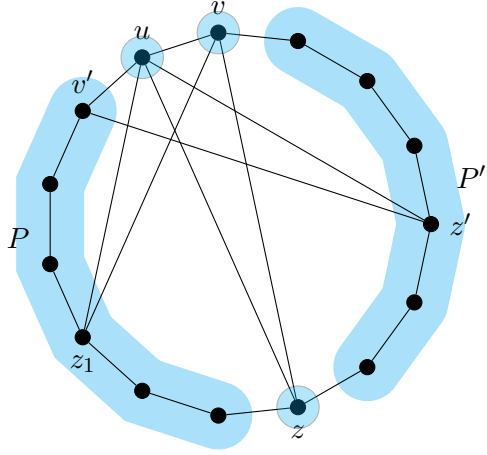
\begin{figure}[!h]
    \centering

    \begin{tikzpicture}[scale=0.4]  
    \tikzset{vertex/.style={circle, minimum size=0.2cm, fill=black, draw, inner sep=1pt}};
  \node (1) at (343.22bp,107.88bp) [vertex] {};
  \node (2) at (299.1bp,47.169bp) [vertex] {};
  \node (3) at (234.11bp,9.6447bp) [vertex, label=below:{$z$}] {};
  \node (4) at (159.47bp,1.8bp) [vertex] {};
  \node (5) at (88.097bp,24.991bp) [vertex] {};
  \node (6) at (32.325bp,75.209bp) [vertex, label=below:{$z_1$}] {};
  \node (7) at (1.8bp,143.77bp) [vertex] {};
  \node (8) at (1.8bp,218.82bp) [vertex] {};
  \node (9) at (32.325bp,287.38bp) [vertex, label={$v'$}] {};
  \node (10) at (88.097bp,337.59bp) [vertex, label={$u$}] {};
  \node (11) at (159.47bp,360.79bp) [vertex, label={$v$}] {};
  \node (12) at (234.11bp,352.94bp) [vertex] {};
  \node (13) at (299.1bp,315.42bp) [vertex] {};
  \node (14) at (343.22bp,254.7bp) [vertex] {};
  \node (15) at (358.82bp,181.29bp) [vertex, label=right:{$z'$}] {};
  \node (16) at (1.5,10.5) [vertex, opacity=0] {};
  \node (17) at (6.3,0) [vertex, opacity=0] {};
  \node (18) at (7.5,12.6) [vertex,opacity=0] {};
  \node (19) at (10,1.3) [vertex, opacity=0] {};
  \node (20) at (14,7) [vertex, label={$P'$}, opacity=0] {};
  \node (21) at (-1,5) [vertex, label={$P$}, opacity=0] {};

 \begin{pgfonlayer}{background}
\draw[cyan!30,line width=9mm,line cap=round] (4.center)--(5.center)--(6.center)--(7.center)--(8.center)--(9.center);
\draw[cyan!30,line width=9mm,line cap=round,line join=round] (2.center)--(1.center)--(15.center)--(14.center)--(13.center)--(12.center);
\end{pgfonlayer}
  
  \draw [] (1) -- (2);
  \draw [] (2) -- (3);
  \draw [] (3) -- (4);
  \draw [] (4) -- (5);
  \draw [] (5) -- (6);
  \draw [] (6) -- (7);
  \draw [] (7) -- (8);
  \draw [] (8) -- (9);
  \draw [] (9) -- (10);
  \draw [] (10) -- (11);
  \draw [] (11) -- (12);
  \draw [] (12) -- (13);
  \draw [] (13) -- (14);
  \draw [] (14) -- (15);
  \draw [] (15) -- (1);
  \draw [] (10) -- (6);
  \draw [] (11) -- (6);
  \draw [] (10) -- (3);
  \draw [] (11) -- (3);
  \draw [] (9) -- (15);  
  \draw [] (10) -- (15);
  \draw[fill=cyan,opacity=0.3] (3) circle(0.7cm);
  \draw[fill=cyan,opacity=0.3] (10) circle(0.7cm);
  \draw[fill=cyan,opacity=0.3] (11) circle(0.7cm);
  %
\end{tikzpicture}
    \caption{Model  of a $K_5$ cyclic minor.}
    \label{fig_cyclick5}
\end{figure}


Note that $f(5)\ge 6$, since the icosahedron has minimum degree~5 and does not contain $K_5$ as a minor, much less a cyclic minor. In general, we have a quadratic bound which is possible to obtain using linkages. A graph is said to be \emph{$k$-linked} if it has at least $2k$ vertices and for every sequence $s_1, \ldots, s_k, t_1, \ldots, t_k$ of distinct vertices there exist disjoint paths $P_1, \ldots, P_k$ such that the ends of $P_i$ are $s_i$ and $t_i$. 

\begin{lemma}\label{l:quadr}
$f(k)=O(k^2)$.
\end{lemma}
\begin{proof}


Let $G$ be a graph such that $\delta(G)\geq  40t$, which means that its average degree is at least $40t$. By \cite{MADER72}, every graph with average degree at least $4t$ has a $(t+1)$-connected subgraph with more than $2t$ vertices. Hence, $G$ has a $(10t+10)$-connected subgraph $F$. In~\cite{THOMAS2005309}, it was proved that if a graph with $n$ vertices and at least $5tn$ edges is $2t$-connected, then it is $t$-linked. Since $F$ is $2t$-connected and has at least $(5t+5)|V(H)|>5t|V(H)|$ edges, $F$ is $t$-linked.

Note that this implies that for any distinct vertices $v_1$, \ldots, $v_t$ of $F$, there exists a cycle in $F$ passing through $v_1$, \ldots, $v_t$ in order: For each $i$, choose a vertex $u_i$ adjacent to $v_{i+1}$ (where $v_{t+1}=v_1$) so that the vertices $u_1$, \ldots, $u_t$, $v_1$, \ldots, $v_t$ are pairwise distinct.  The $t$-linkedness of $F$ implies that $F$ contains pairwise vertex-disjoint paths $P_1$, \ldots, $P_t$, where $P_i$ has ends $u_i$ and $v_i$ for each $i$. The concatenation of these paths gives the desired cycle.

If, given a high constant $c$, $G$ is a graph of minimum degree at least $ck^2$, consider a $k^2$-linked subgraph $H$ of $G$ and choose a matching $M$ of size $\binom{k}{2}$ in $H$. Label the vertices incident with $M$ by labels $v_1$, \ldots, $v_{k(k-1)}$ so that letting $V_i=\{v_{(k-1)(i-1)+1},\ldots, v_{(k-1)(i-1)+(k-1)}\}$, for all distinct $i,j\in \{1,\ldots,k\}$, an edge of $M$ has one end in $V_i$ and the other end of $V_j$. Let $C$ be a cycle in $H$ passing through $v_1$, \ldots, $v_{k(k-1)}$ in order.  Contracting the edges of pairwise vertex-disjoint subpaths of $C$ containing $V_1$, \ldots, $V_k$ gives $K_k$ as a cyclic minor of $G$.
\end{proof}

Note that~Lemma~\ref{l:quadr} shows the existence of a $K_k$ cyclic minor in a graph with quadratic minimum degree, but this may not be the case if the Hamiltonian cycle is prescribed, as in~Lemma~\ref{l:clique}. When only contracting the edges of the Hamiltonian cycle $C$, average degree 3 and 6 give $K_4$ and $K_5$ (respectively) as cyclic minors, but we do not know the exact value for $K_6$. However, no bound on the average degree provides a $K_7$ cyclic minor obtained by contracting edges of $C$. For instance, $C$ can have many crossing chords arranged in a bipartite configuration (see~Fig.~\ref{fig_bipk7}).

\begin{figure}[!h]
    \centering

    \begin{tikzpicture}[scale=0.5]  
    \tikzset{vertex/.style={circle, minimum size=0.2cm, fill=black, draw, inner sep=1pt}};
\node (1) at (343.22bp,107.88bp) [vertex] {};
  \node (2) at (299.1bp,47.169bp) [vertex] {};
  \node (3) at (234.11bp,9.6447bp) [vertex] {};
  \node (4) at (159.47bp,1.8bp) [vertex] {};
  \node (5) at (88.097bp,24.991bp) [vertex] {};
  \node (6) at (32.325bp,75.209bp) [vertex] {};
  \node (7) at (1.8bp,143.77bp) [vertex] {};
  \node (8) at (1.8bp,218.82bp) [vertex] {};
  \node (9) at (32.325bp,287.38bp) [vertex] {};
  \node (10) at (88.097bp,337.59bp) [vertex] {};
  \node (11) at (159.47bp,360.79bp) [vertex] {};
  \node (12) at (234.11bp,352.94bp) [vertex] {};
  \node (13) at (299.1bp,315.42bp) [vertex] {};
  \node (14) at (343.22bp,254.7bp) [vertex] {};
  \node (15) at (358.82bp,181.29bp) [vertex] {};
  \node (16) at (1.5,10.5) [vertex, opacity=0] {};
  \node (17) at (6.3,0) [vertex, opacity=0] {};
  \node (18) at (7.5,12.6) [vertex,opacity=0] {};
  \node (19) at (10,1.3) [vertex, opacity=0] {};

  \draw [] (1) -- (2);
  \draw [] (2) -- (3);
  \draw [] (3) -- (4);
  \draw [] (4) -- (5);
  \draw [] (5) -- (6);
  \draw [] (6) -- (7);
  \draw [] (7) -- (8);
  \draw [] (8) -- (9);
  \draw [] (9) -- (10);
  \draw [] (10) -- (11);
  \draw [] (11) -- (12);
  \draw [] (12) -- (13);
  \draw [] (13) -- (14);
  \draw [] (14) -- (15);
  \draw [] (15) -- (1);
  \draw [] (1) -- (9);
  \draw [] (1) -- (10);
  \draw [] (1) -- (11);
  \draw [] (1) -- (12);
  \draw [] (1) -- (13);
  \draw [] (2) -- (9);
  \draw [] (2) -- (10);
  \draw [] (2) -- (11);  
  \draw [] (2) -- (12);
  \draw [] (2) -- (13);
  \draw [] (3) -- (9);
  \draw [] (3) -- (10);
  \draw [] (3) -- (11);
  \draw [] (3) -- (12);
  \draw [] (3) -- (13);
  \draw [] (4) -- (9);
  \draw [] (4) -- (10);
  \draw [] (4) -- (11);
  \draw [] (4) -- (12);
  \draw [] (4) -- (13);
  \draw [] (5) -- (9);
  \draw [] (5) -- (10);
  \draw [] (5) -- (11);
  \draw [] (5) -- (12);
  \draw [] (5) -- (13);
  
  %
\end{tikzpicture}
    \caption{Cycle with chords in bipartite configuration.}
    \label{fig_bipk7}
\end{figure}
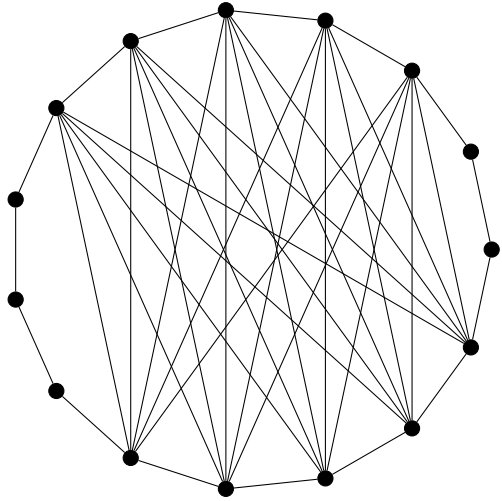


Note that a $K_7$ cyclic minor cannot exist in such a bipartite configuration with partite sets $A$ and $B$. Indeed, there would be three contracted sets forming vertices of $K_7$ which would be entirely in $A$ or entirely in $B$, but then two of them would not be adjacent. Bipartite configurations can also be used to prove that if $C$ has large average degree, then one can form $K_6$ by contracting its edges. To see this, apply the Marcus-Tardos Theorem (as in the beginning of this section) to form a bipartite configuration $X_1,X_2,X_3,X_4,Y_1,Y_2,Y_3,Y_4$ and furthermore contract $Y_4,X_1$ and $X_4,Y_1$ to form $K_6$.



\section{Concluding remarks and open problems}
\label{sec:open}

What we do here is in fact algorithmic, in the sense that we implicitly describe a polynomial-time algorithm whose input is a graph with minimum degree at least $k$ and whose output are cycles satisfying the conclusion of Theorem~\ref{th:main}.  However, finding an optimal lollipop is NP-hard, since solving it in polytime would imply finding a Hamiltonian cycle in polytime. Also, there might be exponentially many active paths in a lollipop. So we need to briefly explain our algorithmic claim. 

In fact we neither need an optimal lollipop nor the set of all active paths to obtain the properties we want. What we need is a lollipop and a set of at least $k$ (sometimes $k+1$) active paths, all with  distinct ends.  Formally, we may compute them as follows (note that we maintain a set $\mathcal A$ of active vertices discovered so far):

\begin{itemize}
    \item    
$\mathcal S_1 = \{c_1 c_2 \dots c_t, c_1c_t c_{t-1} \dots c_2\}$, $\mathcal A = \{c_2, c_t\}$.
\item 
For all $i\geq 1$, let us define $\mathcal S_{i+1}$ from $\mathcal S_{i}$. For each $Q = c_1\dots u \in \mathcal S_i$ and all vertices $v$ such that $uv$ is a chord of $Q$, let $w$ be the neighbor of $v$ in $vQu$. If $vw$ is an edge of $C$ and $w\notin \mathcal A$, then add the path $c_1QvuQw$ to $\mathcal S_{i+1}$ and $w$ to $\mathcal A$. 
\end{itemize}

The proofs of the lemmas from Section~\ref{sec:proofs} can now be seen as the description of a procedure $P$ whose input is a graph $G$ with minimum degree at least $k$, and a lollipop of $G$ that either outputs a better lollipop (with more vertices or a longer cycle), or the cycle described in Theorem~\ref{th:main}. Let us explain some key steps justifying this claim. Each time the end of an active path has neighbors outside of the cycle of the lollipop, a better lollipop can be computed as explained in the proof of Lemma~\ref{lem:activeneighborhood}. Also, for each new active vertex, we test whether its neighborhood is actually in the cycle. All this guarantees that the computation is performed in polynomial time. Note that the proofs of Lemma~\ref{lem:allpath} and Lemma~\ref{lem:nonactivepath} remain correct with these new settings. 

Now, the global algorithm starts with a call to DFS (complexity of $O(n+m)$) to find the first lollipop, and then to the procedure $P$. While the call to $P$ fails to produce the cycles we want, it must be that a better lollipop is discovered, in which case we call $P$ again with the new lollipop. Since each new lollipop has either a longer cycle or more vertices, there are at most $n^2$ calls to $P$.

Back to the original definition of active paths, we have several questions. First, is every Hamiltonian path starting in $c_1$ active? The answer is negative, since in the complete graph $K_6$ with the Hamiltonian cycle $(c_1,\dots ,c_6)$, the path $Q=c_1c_2c_5c_4c_3c_6$ is not active (there are 6 non active paths among the 120 paths starting by $c_1$). To see this, note that $Q$ could be either generated from $Q_1=c_1c_6c_3c_4c_5c_2$ or from $Q_2=c_1c_2c_5c_6c_3c_4$. But $Q_2$ can only be generated from $Q$, and the only possible generator for $Q_1$ (apart from $Q$) is $c_1c_6c_3c_2c_5c_4$ which can only be generated from $Q_1$. This fact leads to the two following questions:

\begin{ques}
Is there a simple criterion for a sequence of vertices to be an active path of $K_n$? 
\end{ques}

\begin{ques}
Is there a simple way to compute all the active vertices in the cycle of an optimal lollipop? 
\end{ques}

High minimum degree indeed provides cycles with many chords but one may wonder whether a similar statement could still hold for graphs with minimum degree 3. Since the disjoint union of many copies of $K_4$ does not contain a cycle with more than two chords, a natural question is to consider graphs with high girth $g$. It turns out that in the lollipop argument, the number of distinct active vertices (and thus chords) increases with respect to $g$. It is then reasonable to link the number of chords to the length of the host cycle. This leads to the following question:

\begin{ques} 
Is there a constant $c>0$ such that every graph with minimum degree 3 contains a cycle of length $\ell$ with at least $c \ell$ chords? 
\end{ques}

We are very far from answering the previous problem, as even the following relaxation seems unclear:

\begin{ques}
Is there a function $f$ that tends to $+\infty$ such that every graph with minimum degree 3 contains a cycle of length $\ell$ with at least $f(\ell)$ chords?
\end{ques}

    Note that the following question can be entirely solved: does sufficiently large minimum degree imply the existence of a cycle with \emph{exactly} $\ell$ chords? The answer is positive for $\ell =0$, as one can consider any induced cycle. Then the answer becomes negative for all $\ell=1,\dots ,34$, and becomes positive again for $\ell=35$. To explain this, note that cliques only contain cycles whose number of chords are expressible as $a(a-3)/2$ for an integer $a$, and bicliques only contain cycles whose number of chords are expressible as $b^2-2b$ for an integer~$b$. So the only candidates for $\ell$ are values both expressible as $a(a-3)/2$ and as $b^2-2b$. It turns out that all these candidate values give positive answers. Indeed,  it has been announced in \cite{pc} that  a graph of sufficiently large minimum degree contains $K_a$, $K_{b,b}$, or an induced cycle~$C$ and a vertex $v$ with at least $\ell+3$ neighbors in $C$. In this latter case, $v$ together with a subpath of $C$ forms a cycle with exactly $\ell$ chords. The smallest positive candidate $\ell$ is 35, both corresponding to the number of chords in the hamiltonian cycle of  $K_{10}$ and of $K_{7,7}$.  Let us remark that there are infinitely many values of $\ell$ with the described property, since they correspond to solutions to a Pell's equation.

\subsection*{Chromatic number and chords}

In the context of bounding the chromatic number, chords in cycles also attracted some attention. The contributions in this line of research make no reference to~\cite{GUPTA198037}. So it might be useful to  list  the consequences of Theorem~\ref{th:main} regarding the chromatic number and put them in the right context. 

In~\cite{mft:chordless}, a structure theorem for graphs where no cycle has a chord is described. A stronger theorem (where only cycles with a unique chord are excluded) is presented in~\cite{nicolas.kristina:one}, and it is generalized to graphs where no cycle of length at least~5 has a unique chord in~\cite{DBLP:journals/jgt/TrotignonP18}. All these structural descriptions imply upper bounds on the chromatic number of the graphs under consideration. In \cite{DBLP:journals/dam/AboulkerB15}, it is shown that graphs with no cycle with exactly two chords are 6-colorable, while graphs with no cycle with exactly three chords satisfy bounds on their chromatic number. Also bounds on the chromatic number of graphs containing no cycles with exactly $k$ chords are conjectured, and they are proved in~\cite{leeLP:chords} for sufficiently large~$k$. From Theorem~\ref{th:main}, it is possible to obtain the following Corollary. A graph is \emph{$k$-degenerate} if all its subgraphs contain a vertex of degree at most~$k$. 

    \begin{cor}
        \label{cor:degenerate}
     For all integers $\ell \geq 0$, if every cycle of a graph $G$ has less than $\ell$ chords, then $G$ is $\left\lceil\frac{-1+\sqrt{9+8\ell}}{2}\right\rceil$-degenerate and therefore $\left\lceil\frac{1+\sqrt{9+8\ell}}{2}\right\rceil$-colorable.
    \end{cor}

    \begin{proof}
      Suppose for a contradiction that $G$ is not $\left\lceil\frac{-1+\sqrt{9+8\ell}}{2}\right\rceil$-degenerate. So, $G$ has a subgraph with minimum degree at least $\lceil k\rceil$, where $$k = \frac{-1+\sqrt{9+8\ell}}{2} + 1 = \frac{1+\sqrt{9+8\ell}}{2}.$$ Hence, by Theorem~\ref{th:main}, $G$ contains a cycle whose number of chords is at least:  
      $$\frac{(\lceil k\rceil +1)(\lceil k\rceil - 2)}{2} \geq \frac{( k +1)(k - 2)}{2}  = \ell.$$ This contradiction proves the claim about the degeneracy and the claim about colorability follows.  
    \end{proof}
    

    Note that for $\ell=0$, Corollary~\ref{cor:degenerate} restates a well-known fact: every forest is 1-degenerate. For $\ell=1$, it states that graphs where all cycles are chordless are 2-degenerate, a result already present in~\cite{mft:chordless}.  Observe also that Corollary~\ref{cor:degenerate} is tight for all $\ell \geq 0$. Indeed, for $k\geq 1$, set $$I_k = \left\{ \ell \,\middle|\, \left\lceil\frac{-1+\sqrt{9+8\ell}}{2}\right\rceil = k\right\}.$$ Every $\ell$ is in some $I_k$. A simple computation shows that $$I_k = \left\{\ell \,\middle|\, \frac{(k+1)(k-2)}{2} < \ell \leq \frac{(k+2)(k-1)}{2}\right\}.$$ So, for any $\ell\in I_k$, Corollary~\ref{cor:degenerate} says that a graph with all cycles having less than $\ell$ chords is $k$-degenerate, and this is best possible since every cycle in $K_{k+1}$ has at most $\frac{(k+1)(k-2)}{2} < \ell$ chords, while $K_{k+1}$ is not $(k-1)$-degenerate.

\section*{Acknowledgement}

The authors thank Matthias Thomassé for the non active path of $K_6$ and Dan Kr\'al' for pointing out to us~\cite{GUPTA198037} and~\cite{kral03}.


\end{document}